\documentclass[12pt, one sided]{amsart}
\RequirePackage[colorlinks,citecolor=blue,urlcolor=blue]{hyperref}

\usepackage{latexsym,amssymb,amsmath,amsfonts,amsthm}
\usepackage{amsthm,amsmath}
\usepackage{float}
\usepackage{enumerate}
\usepackage[margin=3cm]{geometry}
\usepackage{bbm}
\usepackage{subfigure}
\usepackage{graphicx}
\usepackage[toc,page]{appendix}
\usepackage{amsfonts}
\usepackage[all]{xy}
\usepackage{caption}
\usepackage{geometry}                
\usepackage{graphicx}
\usepackage{amssymb}
\usepackage{epstopdf}
\usepackage{color}
\usepackage{bbm, dsfont}
\usepackage{hyperref}
\usepackage{cleveref}
\usepackage[utf8]{inputenc}
\usepackage{amssymb,amsthm,amsmath,amssymb,wrapfig,dsfont}
\usepackage[dvipsnames]{xcolor}
\definecolor{myred}{RGB}{251,154,133}
\definecolor{myblue}{RGB}{153,206,227}
\definecolor{mylightblue}{RGB}{0, 150, 255}
\definecolor{mygreen}{RGB}{32, 210, 64}
\definecolor{mygray}{RGB}{220, 220, 220}

\usepackage{tikz}
\usetikzlibrary{decorations.pathmorphing}
\tikzset{snake it/.style={decorate, decoration=snake}}
\usetikzlibrary{shapes.geometric,positioning,decorations.pathreplacing} 

\usepackage{breqn}
\usepackage{upgreek}

\def\beq{ \begin{equation} }
	\def\eeq{ \end{equation} }

\def\ms{\medskip}

\def\square{\vcenter{\vbox{\hrule height .4pt
			\hbox{\vrule width .4pt height 5pt \kern 5pt
				\vrule width .4pt} \hrule height .4pt}}}

\newcommand{\BZ}{{\mathbb{Z}}}

\newcommand{\e}{{\bf E}}
\newcommand{\bae}{\begin{equation}\begin{aligned}}
\newcommand{\eae}{\end{aligned}\end{equation}}
\newcommand{\ev}{\mathbf{E}}

\DeclareFontFamily{OML}{rsfs}{\skewchar\font'177}
\DeclareFontShape{OML}{rsfs}{m}{n}{ <5> <6> rsfs5 <7> <8> <9>
	rsfs7 <10> <10.95> <12> <14.4> <17.28> <20.74> <24.88> rsfs10 }{}
\DeclareMathAlphabet{\mathfs}{OML}{rsfs}{m}{n}

\newcounter{relctr} 
\everydisplay\expandafter{\the\everydisplay\setcounter{relctr}{0}} 

\makeatletter

\pdfpageheight\paperheight
\pdfpagewidth\paperwidth

\numberwithin{equation}{section}
\theoremstyle{plain}
\newtheorem{thm}{\protect\theoremname}
\theoremstyle{plain}
\newtheorem{prop}[thm]{\protect\propositionname}
\theoremstyle{plain}
\newtheorem{lem}[thm]{\protect\lemmaname}

\usepackage{tikz}

\makeatother

\usepackage{babel}
\providecommand{\lemmaname}{Lemma}
\providecommand{\propositionname}{Proposition}
\providecommand{\theoremname}{Theorem}

\begin{document}
\title{A toy model for DLA arm growth in a wedge}
\author{Oren Louidor}
\thanks{The first and last authors are supported by BSF grant 2018330. The first author is supported by ISF grants 1382/17 and 2870/21. The last author is supported by the Chaya career advancement chair. The second author is supported by ISF grant 765/18}
	\address[Oren Louidor]{Technion - Israel Institute of Technology}
	\urladdr{https://iew.technion.ac.il/~olouidor/}
	\email{oren.louidor@gmail.com }
\author{Chanwoo Oh}
	\address[Chanwoo Oh]{Technion - Israel Institute of Technology}
	\email{chanwoo.oh@outlook.com}
\author{Eviatar B. Procaccia}
	\address[Eviatar B. Procaccia]{Technion - Israel Institute of Technology}
	\urladdr{http://procaccia.net.technion.ac.il}
	\email{procaccia@technion.ac.il}
\begin{abstract}
In this paper, we consider a non-homogeneous discrete-time Markov chain which can be seen as a toy model for the growth of the arms of the DLA (Diffusion limited aggregation) process in a sub-linear wedge. It is conjectured that in a thin enough linear wedge  there is only one infinite arm in the DLA cluster and we demonstrate this phenomenon in our model. The technique follows a bootstrapping argument, in which we iteratively prove ever faster growth rate.
\end{abstract}

\maketitle
\global\long\def\mc#1{\mathcal{#1}}%

\global\long\def\ms#1{\mathscr{#1}}%

\global\long\def\mb#1{\mathbb{#1}}%

\global\long\def\mf#1{\mathfrak{#1}}%

\global\long\def\Cap{\text{cap}}%

\global\long\def\l{\ell}%

\global\long\def\z{\zeta}%

\global\long\def\a{\alpha}%

\global\long\def\b{\beta}%

\global\long\def\d{\delta}%

\global\long\def\e{\varepsilon}%

\global\long\def\g{\gamma}%

\section{Introduction}

The Diffusion Limited Aggregation (DLA) model was introduced by Witten and Sander \cite{witten1981diffusion}, as a simple model for various natural phenomena such as electro-chemical deposition, ion-aggregation and even bacterial growth patterns. DLA is defined by iteratively bringing random walks from infinity, via the discrete harmonic measure, and connecting the last vertex and edge traversed by the random walk before hitting the current aggregate. Although the model was defined in the 1980's, many of the most basic questions remain unanswered. No proof currently exists that DLA has fractal behavior, and no non-trivial lower bound for the growth rate is known. Kesten \cite{kesten1987} proved via a discrete version of the Beaurling projection theorem, that after $n$ particles, the arms of DLA are asymptotically at most at distance $n^{2/3}$. Recently \cite{berger2021growth,procaccia2021dimension}, Berger, Procaccia and Turner, established an exact fractal dimension result, mathching to Kesten's upper bound for an off lattice variant of DLA called Stationary Hastings-Levitov$(0)$. 

One of the long standing conjectures is that in thin enough linear wedge there exists only one infinite arm in a DLA cluster \cite{kessler1998diffusion,Stabilization_of_DLA}. We define the number of arms as the number of topological ends, i.e. let $B(r)=\{x\in\BZ^2: \|x\|_2 \le r \}$ and let $A_n$ be the DLA process after attaching $n$ particles. Since $\left( A_n \right)_{n\geq 0}$ is increasing, it has a well defined limit $A_\infty$. The number of topological ends is 
$$
\limsup_{r\to\infty}\#\{\text{connected components of }A_\infty\setminus B(r)\} 
.$$

In this work, we consider a non-homogeneous discrete-time Markov chain, which serves as a toy model for the growth of DLA arms in a sub-linear wedge. We show that this process becomes a monotone increasing sequence after a time which is finite almost surely by means of a bootstrapping argument, which corresponds to the existence of only one infinite DLA arm.
Formally, let $D_{n}$ be a Markov chain with $D_{0}=0,\,D_{1}=1$. For $n\geq1$,
$D_{n+1}$ is given as follows: if 
$D_{n}\neq0$, then 
\begin{equation}
\label{e:1}
D_{n+1}=\begin{cases}
D_{n}+1, & {\rm with\,probability}\,1-\frac{1}{2}e^{-C\frac{D_{n}}{n^{\alpha}}}\\
D_{n}-1, & {\rm with\,probability}\,\frac{1}{2}e^{-C\frac{D_{n}}{n^{\alpha}}}
\end{cases},
\end{equation}
and if $D_{n}=0$, then $D_{n+1}=1$. Here, $0<\alpha<1$ and $C>0$.

To justify the connection to DLA, consider the sub-linear wedge $$W_{\alpha}=\left\{ \left(x,y\right)\in\mb Z^{2}\,:\,0\leq y\leq x^{\alpha},\,x\geq0\right\},$$ where $\alpha\in\left(0,1\right)$ and let $\left(A_{n}\right)_{n\geq0}$ be the DLA process on $W_\alpha$. The latter is defined by setting $A_{0}=\left\{ \left(0,0\right)\right\} $ and, for $n \geq 1$,  letting $A_{n}=A_{n-1}\cup\left\{ a_{n}\right\}$, where $a_{n}$ is the last vertex visited by a simple random walk which starts from infinity, remains confined to the wedge, and stopped when it hits $A_{n-1}$. For details of the construction, refer to \cite{Stabilization_of_DLA}.

\begin{figure}[h!]
\begin{centering}
\includegraphics[height=3cm]{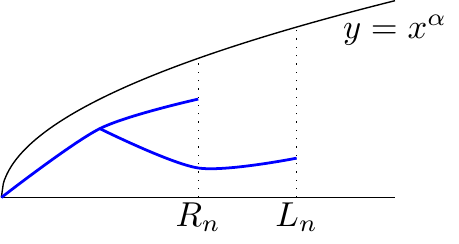}\caption{Diffusion limited aggregation in a sub-linear wedge with two longest
arms}
\par\end{centering}
\end{figure}

Now, suppose that $L_n\ge R_n$ are the $x$-coordinates of the tips of the two longest arms of $A_n$, defined in some natural way (e.g. one may take the maximal $x$ for which the aggregate contains at least $1$, respectively $2$, distinct vertices with $x$ as their horizontal coordinate). The probability that a random walk as above  reaches the shorter arm before it hits the longer one, is approximately the probability that a planar random walk which starts at the origin, hits the line $\{x=\tilde{D}_n := L_n-R_n\}$ before hitting any of the two half lines $\{y= 0, \, x \geq 0\}$ and $\{y= L_n^\alpha, \, x \geq 0\}$. Dividing the tube $[0,\tilde{D}_n] \times \{0, L^{\alpha}\}$ into $\tilde{D}_n/L_n^\alpha$ squares of side length $L_n^\alpha$, and using the invariance principle for planar random walks, it is not difficult to see that the latter crossing probability decays exponentially in $\tilde{D}_n/L_n^\alpha$. If we further stipulate that most of time growth happens at the tip of an arm -- an assumption not uncommon in the Physics literature for approximating the growth of DLA (c.f., the Laplacian growth models~\cite{barra2001laplacian,gustafsson2014classical,hastings1998laplacian}), then $L_n$ is proportional to $n$ and the crossing probability becomes exponential in $\tilde{D}_n/n^\alpha$, taking a form as in~\eqref{e:1}. 

We remark that in order for the result of this paper to translate to the arms of DLA, it  would be enough to prove a variant of Proposition \ref{prop:MainProp} with $\tilde{D}_n=L_n-R_n$ replacing the Markov chain $D_n$. Unfortunately, making a precise statement based on the above considerations for DLA remains open.




\section{Results}
Let $(D_n)_{n\geq0}$ be defined as above \eqref{e:1}. The following theorem is the main result of this paper.
\begin{thm}
\label{thm:MainThm}For any $\a\in(0,1)$ and $C>0$, $$\mb P\left(\exists\,N\in\mb N\textup{ such that }\forall\,n\geq N,\,D_{n+1}=D_{n}+1\right)=1.$$
\end{thm}
The main proof idea is using Doob's decomposition theorem to decompose this process into a martingale and a predictable process. The latter is an increasing process since this Markov chain is a submartingale. We get an asymptotic growth bound for the martingale using Azuma-Hoeffding inequality. At each step, we get better asymptotic growth bounds for the increasing process and subsequently for the Markov chain. 

We will use the following proposition to prove this theorem.
\begin{prop}
\label{prop:MainProp}For any $\varepsilon>0$, there exist $\beta$
with $\alpha<\beta<1$ and a natural number $N$, and $s>0$ such that
\[
\textrm{for any }n\geq N,\,\,\mathbb{P}\left(D_{n}\geq s n^{\beta}\right)\geq1-\varepsilon.
\]
\end{prop}

\section{Proof of Proposition~\ref{prop:MainProp}} 

%


For $\e>0$, $\d\in(0,\frac{1}{2})$, and $n\in\mathbb{N}$, define the event
\[
E_{1,n,\e,\d}:=\left\{ \frac{\left|\left\{ 0<m\leq n\,:\,D_{m}\geq m^{\frac{1}{2}-\delta}\right\} \right|}{n} \geq 1-\e\right\} ,
\]
 where $\left|\cdot\right|$ represents the number of elements of
a set. Since $\left( D_n \right)_{n\geq 0}$ is clearly a sub-martingale with respect to the natural filtration, we may consider its Doob Decomposition $D_{n}=A_{n}+M_{n}$, where $\left( A_n \right)_{n\geq 1}$ is predictable and increasing, and $\left( M_n \right)_{n\geq 1}$ is a martingale. For $\e>0$, $\bar{\delta}\in(0,\frac{1}{2})$, and $n\in\mathbb{N}$,
define the event
\[
E_{2,n,\e,\bar{\delta}}:=\left\{ \left|M_{m}\right|\leq m^{\frac{1}{2}+\bar{\delta}}\textrm{ for all }2\e n\leq m\leq n\right\} .
\]
Also, let $E_{n,\e,\d,\bar{\delta}}:=E_{1,n,\e,\d}\cap E_{2,n,\e,\bar{\delta}}$.

The following lemmas provide probability bounds for $E_{1,n,\e,\d}$
and $E_{2,n,\e,\bar{\delta}}$.
\begin{lem}
\label{lem:E1}For any $\e>0$ and $\d\in(0,\frac{1}{2})$, there exists a natural number $N$
such that
\[
\textrm{for any }n\geq N,\,\,\mathbb{P}\left(E_{1,n,\e,\d}\right)\geq1-\e.
\]
\end{lem}

\begin{proof}
Let $S_{n}$ be the simple symmetric random walk on $\mathbb{Z}$.
Let 
\[
R_{1,n,\e,\d}:=\left\{ \frac{\left|\left\{ 0<m\leq n\,:\,\left|S_{m}\right|\geq m^{\frac{1}{2}-\delta}\right\} \right|}{n} \geq 1-\e\right\} .
\]
Since for $D_n \neq 0$, $\mathbb{P} (D_{n+1} = D_n + 1 ) \geq \frac{1}{2}$ and $\mathbb{P} (D_{n+1} = D_n - 1 ) \leq \frac{1}{2}$,
by coupling $D_{n}$ and $\left|S_{n}\right|$, we get  $\mathbb{P}\left(E_{1,n,\e,\d}\right)\geq\mathbb{P}\left(R_{1,n,\e,\d}\right)$.
Also,
$$ \mathbb{P}\left(R_{1,n,\e,\d}\right) \geq \mathbb{P}\left(\left\{ \frac{\left|\left\{ 0<m\leq n\,:\,\left|S_{m}\right|\geq n^{-\d}n^{\frac{1}{2}}\right\} \right|}{n} \geq 1-\e\right\} \right).$$

By Donsker's invariance theorem, $\frac{\left|\left\{ 0<m\leq n\,:\,\left|S_{m}\right|\geq an^{\frac{1}{2}}\right\} \right|}{n}$
converges in distribution to $\lambda\left(\left\{ t\in[0,\,1]\,:\,\left|B_{t}\right|\geq a\right\} \right)$
for any $a>0$ as $n\rightarrow\infty$, where $\lambda$ is the
Lebesgue measure on the real line (for details, we refer to Example~8.6.4.
Occupation times of half-lines of \cite{durrett2010probability}.)
We can take $a>0$ such that $\mb P\left(\lambda\left(\left\{ t\in[0,\,1]\,:\,\left|B_{t}\right|\geq a\right\} \right) \geq 1-\e\right)\geq1-\frac{\e}{2}$,
since $\mb P\left(\lambda\left(\left\{ t\in[0,\,1]\,:\,\left|B_{t}\right|\geq a\right\} \right) \geq 1-\e\right)\rightarrow1\text{ as }a\searrow0$.
Take $N_{1},\,N_{2}$ to satisfy $N_{1}^{-\d}<a$ and
\begin{align*}
\left|\mathbb{P}\left(\left\{ \frac{\left|\left\{ 0<m\leq n\,:\,\left|S_{m}\right|\geq an^{\frac{1}{2}}\right\} \right|}{n} \geq 1-\e\right\} \right)\,\,\,\,\,\,\,\,\,\,\,\right.\\
\left.-\mb P\left(\vphantom{\frac{\left|\left\{ 0<m\leq n\,:\,\left|S_{m}\right|\geq an^{\frac{1}{2}}\right\} \right|}{n}}\lambda\left(\left\{ t\in[0,\,1]\,:\,\left|B_{t}\right|\geq a\right\} \right) \geq 1-\e\right)\right| & \leq\frac{\e}{2}
\end{align*}
 for all $n\geq N_{2}$. Let $N=\max\left(N_{1},N_{2}\right)$. Then,
for $n\geq N$ we have $\mathbb{P}\left(E_{1,n,\e,\d}\right)\geq1-\e.$
\end{proof}
\begin{lem}
\label{lem:E2}For any $\e>0$ and $\bar{\delta}\in(0,\frac{1}{2})$, there exists a natural number $N$
such that
\[
\textrm{for any }n\geq N,\,\,\mathbb{P}\left(E_{2,n,\e,\bar{\delta}}\right)\geq1-\e.
\]
\end{lem}

\begin{proof}
Since $M_{m}=D_{m}-A_{m}$,
$$\left|M_{m+1}-M_{m}\right|\leq\Big|D_{m+1}-\ev[D_{m+1}|\mathfs{F}_m]\Big|\leq \Big|D_{m}+1-(D_{m}-1)\Big|=2$$
for every natural number $m$. By Azuma-Hoeffding inequality,
$$\mb P\left(\left|M_{m}\right|>m^{\frac{1}{2}+\bar{\d}}\right)\leq2\textrm{\,exp}\left(\frac{-m^{1+2\bar{\d}}}{8m}\right)=2\textrm{\,exp}\left(-\frac{1}{8}m^{2\bar{\d}}\right).$$
Then,
\begin{align*}
\mathbb{P}\left(E_{2,n,\e,\bar{\delta}}\right) & =\mb P\left(\left|M_{m}\right|\leq m^{\frac{1}{2}+\bar{\delta}}\,\,\textrm{for all }2\e n\leq m\leq n\right)\\
 & \geq1-\sum_{m=\left\lceil 2\e n\right\rceil }^{n}\mb P\left(\left|M_{m}\right|>m^{\frac{1}{2}+\bar{\delta}}\right)\,\,\geq1-\sum_{m=\left\lceil 2\e n\right\rceil }^{n}2\,\textrm{exp}\left(-\frac{1}{8}m^{2\bar{\d}}\right)\\
 & \geq1-2n\,\textrm{exp}\left(-\frac{1}{8}\left\lceil 2\e n\right\rceil ^{2\bar{\d}}\right).
\end{align*}
 As $n\rightarrow\infty$, $n\,\textrm{exp}\left(-\frac{1}{8}\left\lceil 2\e n\right\rceil ^{2\bar{\d}}\right)\rightarrow0$.
So there exists $N\in\mb N$ such that, for any $n\geq N$, $2n\,\textrm{exp}\left(-\frac{1}{8}\left\lceil 2\e n\right\rceil ^{2\bar{\d}}\right)\leq\e$.
Thus we get $\mathbb{P}\left(E_{2,n,\e,\bar{\delta}}\right)\geq1-\e$ for any $n\geq N$. 
\end{proof}
The following lemma is the key for the proof of the Proposition~\ref{prop:MainProp}.
The proof is by induction. At the $k$-th step of the induction,
we get an improved bound for $D_{m}$ for all $\left(k+1\right)\e n\leq m\leq n$.
Consider the drift term $A_m$ in the Doob Decomposition. If $D_{i}=0$, then $D_{i+1}=1$ and $A_{i+1}-A_{i}=1$. If $D_{i}\neq0$, then $A_{i+1}-A_{i}=\ev[ D_{i+1} - D_i | \mathfs{F}_i ] =\left(1-\frac{1}{2}e^{-C\frac{D_{i}}{i^{\alpha}}}\right)-\left(\frac{1}{2}e^{-C\frac{D_{i}}{i^{\alpha}}}\right)=1-e^{-C\frac{D_{i}}{i^{\alpha}}}$. So $A_m \geq \sum_{i=1}^{m-1}1-e^{-C\frac{D_{i}}{i^{\alpha}}}$.
To get a lower bound on $A_{m}$, we use the bound on $D_{i}$ from step $k-1$.
We collect the increments, $1-e^{-C\frac{D_{i}}{i^{\alpha}}}$,
for a long enough time larger than $k \e n$.
Combining this with the bound for $M_{m}$ in Lemma.~\ref{lem:E2},
we obtain a better bound for $D_{m}$ for all $\left(k+1\right)\e n\leq m\leq n$.

For the following lemma, we set $\delta>0$ and $K\in\mathbb{N}$ as follows: Take  $K=\max \left( \left\lfloor \frac{\alpha-\frac{1}{2}}{1-\alpha}\right\rfloor + 1,\, 1 \right)$. If $\a \leq \frac{1}{2}$, then take $\delta=\frac{1}{2} (1-\a)$.
If $\alpha>\frac{1}{2}$ and $\frac{\alpha-\frac{1}{2}}{1-\alpha}\in\mathbb{N}$,
then take $\delta=\frac{1-\alpha}{8}$. If $\alpha>\frac{1}{2}$ and $\frac{\alpha-\frac{1}{2}}{1-\alpha}\notin\mathbb{N}$,
then take $$\delta=\textrm{min}\left(\text{\ensuremath{\frac{1}{8}\left(\frac{1}{2}+K\left(1-\alpha\right)-\alpha\right)},\,\ensuremath{\frac{1-\alpha}{8}}}\right).$$ Then, the following holds: (1) $0<\delta<1/2$, (2) $\frac{1}{2}-\delta+\left(1-\alpha\right)>\frac{1}{2}$,
(3) $\frac{1}{2}-\delta+K\left(1-\alpha\right)>\alpha$, and (4) $\frac{1}{2}-\delta+\left(K-1\right)\left(1-\alpha\right)<\alpha$.

\begin{lem}
\label{lem:Bootstrap} If $\d$ and $K$ are as above, then for any $\e>0$ and any natural number $1\leq k\leq K$ there exists a natural number $N_{k}$ and a real number $s_{k}>0$
such that for any $n \geq N_{k}$, $$\mathbb{P}\left( \left\{ D_{m}\geq s_{k}m^{\frac{1}{2}-\d+k\left(1-\a\right)}\,\textrm{for all }\left(k+1\right)\e n\leq m\leq n\right\} \right)\geq1-2\e.$$
\end{lem}

\begin{proof}
We use induction to prove that for each $1\leq k\leq K$ there exist a natural number $N_{k}$ and a real number $s_{k}>0$ such that for any $n \geq N_{k}$,  $\mathbb{P}\left(E_{n,\e,\d,\bar{\delta}}\right)\geq1-2\e$ and
$$ E_{n,\e,\d,\bar{\delta}}\subset\left\{ D_{m}\geq s_{k}m^{\frac{1}{2}-\d+k\left(1-\a\right)}\,\textrm{for all }\left(k+1\right)\e n\leq m\leq n\right\},$$
where $\bar{\delta} = \frac{1}{8}\left(1-\alpha-\delta\right)$.

We show the base case and the inductive step together. If $k>1$, assume the induction hypothesis holds for the index $k-1$. If $k=1$, by Lemma~\ref{lem:E1} and Lemma~\ref{lem:E2},
there exists $N_0 \in\mb N$ satisfying that $\mathbb{P}\left(E_{n,\e,\d,\bar{\delta}}\right)\geq1-2\e$ for any $n \geq N_0$. Suppose $N\geq N_{k-1}$ and $\left(k+1\right)\e n\leq m\leq n$, and
focus on the event $E_{n,\e,\d,\bar{\delta}}$.

First we show that for some positive real number $s_{k,1}$,
\begin{align} \label{eq:driftLB}
	A_m \geq \sum_{i=\left\lceil k \e n \right\rceil}^{m-1} s_{k,1} i^{\frac{1}{2}-\d+(k-1)\left(1-\a\right)-\a}.
\end{align}
If $k=1$,
\bae
A_{m} & =\sum_{i=0}^{m-1}\left(A_{i+1}-A_{i}\right) \quad \geq\sum_{1\leq i\leq m-1,\, 		D_{i}\geq i^{\frac{1}{2} - \delta}} 1-e^{-C\frac{D_{i}}{i^{\alpha}}} \\
& \geq\sum_{1\leq i\leq m-1,\, 		D_{i}\geq i^{\frac{1}{2} - \delta}} 1-e^{-C i^{\frac{1}{2} -\d-\a}} \\
& \geq\sum_{1\leq i\leq m-1,\, 		D_{i}\geq i^{\frac{1}{2} - \delta}} s_{1,1}i^{\frac{1}{2}-\delta-\a}\,,\,\textrm{ with }s_{1,1}=\frac{C}{2}\\
& \geq\sum_{i=\left\lceil \e n\right\rceil} ^{m-1}s_{1,1}i^{\frac{1}{2}-\delta-\a}\,,\,\textrm{ since it's on }E_{1,n,\e,\d}.
\eae
If $k>1$,
\bae
A_{m} & =\sum_{i=0}^{m-1}\left(A_{i+1}-A_{i}\right) \quad \geq\sum_{i=\left\lceil k \e n\right\rceil }^{m-1}1-e^{-C\frac{D_{i}}{i^{\alpha}}}\\
& \geq\sum_{i=\left\lceil k \e n\right\rceil }^{m-1}1-e^{-C\frac{s_{k-1} i^{\frac{1}{2}-\d+(k-1)\left(1-\a\right)}}{i^{\alpha}}},\text{ by the induction hypothesis}\\
& \geq\sum_{i=\left\lceil k \e n\right\rceil }^{m-1}s_{k,1}i^{\frac{1}{2}-\d+(k-1)\left(1-\a\right)-\a}\,,\,\textrm{ with }s_{k,1}= \frac{C s_{k-1}}{2}.\\
\eae
Then the summation in \ref{eq:driftLB} is greater than or equal to
\bae
& \int_{\left\lceil k \e n\right\rceil }^{m}s_{k,1}x^{\frac{1}{2}-\d+(k-1)\left(1-\a\right)-\a}\,dx\\
\geq \,\, &  s_{k,2}m^{\frac{1}{2}-\d+ k \left(1-\a\right)},\text{ for some constant \ensuremath{s_{k,2}>0} and \ensuremath{n\geq N_{k,1}}},\\
& \text{ where \ensuremath{N_{k,1}} is a sufficiently large natural number, since \ensuremath{\left(k+1\right)\e n\leq m}}.
\eae

Assume $n\geq N_{k,1}$ as well. Then $D_{m}=A_{m}+M_{m}\geq s_{k,2}m^{\frac{1}{2}-\d+ k \left(1-\a\right)}-m^{\frac{1}{2}+\bar{\d}}$, since it's
on $E_{2,n,\e,\bar{\delta}}$.
As $\frac{1}{2}+\bar{\d}<\frac{1}{2}-\d+ k\left(1-\a\right)$,
there exist some $N_{k,2}\in\mb N$ and $s_{k,3}>0$ such that
if $n\geq N_{k,2}$ as well then $D_{m}\geq s_{k,3}m^{\frac{1}{2}-\d+ k  \left(1-\a\right)}$. By taking $N_{k}=\text{max }\left\{ N_{k-1},N_{k,1},N_{k,2}\right\} $,
we complete the base case and the inductive step, which prove this lemma.\end{proof}

The proof of Proposition~\ref{prop:MainProp} follows directly by
using Lemma~\ref{lem:Bootstrap}, using $\frac{\e}{2}$, and setting $k=K$, $m=n$, $\b=\frac{1}{2}-\d+K\left(1-\a\right)$,
$N=N_{K}$, and $s=s_{K}$. We have $\a<\beta<1$, $N$, and $s>0$ satisfying
the condition of Proposition~\ref{prop:MainProp}.

\section{Proof of Theorem~\ref{thm:MainThm}}

We prove the following statement: 
\begin{align}\label{eq:infgrowth}
\forall \e>0, \, \exists N\in\mb N, \, \mathbb{P}\left(\textrm{for all }n\geq N,\,D_{n+1}=D_{n}+1\right)\geq1-2\e.
\end{align}
We can assume $s=1$ in using Proposition~\ref{prop:MainProp} by comparing $sn^\beta$ and $n^{\bar{\beta}}$ for  $\a<\bar{\beta}<\beta.$
By Proposition~\ref{prop:MainProp},
$$ \exists N_0 \in \mb N, \,  \forall n\geq N_0, \, \mathbb{P}\left(D_{n}\geq n^{\beta}\right)\geq1-\varepsilon,$$ 
for some $\beta$ with $\alpha<\beta<1$. Also,
$$ \exists N_1 \in \mb N, \,  \forall n\geq N_1,  \, 1-\frac{1}{2}\left(e^{-Cn^{\b-\a}}\right)\geq\exp\left(-e^{-Cn^{\b-\a}}\right),$$
and since $\sum_{k=1}^\infty e^{-C k^{\b - \a}} < \infty$,
$$ \exists N_2 \in \mb N, \,  \forall n\geq N_2, \,  \exp\left(-\sum_{k=n}^{\infty}e^{-C k^{\b - \a}}\right) \geq \frac{1-2\e}{1-\e}. $$

Let $N=\max\left(N_{0},N_{1},N_{2}\right)$. Then,
\bae
 & \mb P\left(\textrm{for all }n\geq N,\,D_{n+1}=D_{n}+1\right)\\
\geq \,\, & \mathbb{P}\left(D_{N}\geq N^{\beta},\,D_{N+1}=D_{N}+1,\,D_{N+2}=D_{N+1}+1,\,D_{N+3}=D_{N+2}+1,\,\cdots\right).
\eae
 Also,
\bae
 & \mathbb{P}\left(D_{N+1}=D_{N}+1,\,D_{N+2}=D_{N+1}+1,\,D_{N+3}=D_{N+2}+1,\,\cdots|\,D_{N}\geq N^{\beta}\right)\\
=\, &  \prod_{k=0}^\infty \mb P\left(D_{N+k+1}=D_{N+k}+1 \, | \, D_{N}\geq N^{\beta},\bigcap_{j=0}^{k-1}\{D_{N+j+1}=D_{N+j}+1\}\right)\\
\geq \,\,& \prod_{k=0}^\infty \left(1-\frac{1}{2}e^{-C\frac{N^{\b}+k}{(N+k)^{\a}}}\right)\\
\geq \,\,& \prod_{k=0}^\infty \left(1-\frac{1}{2}e^{-C (N+k)^{\b-\a}} \right) \\
\geq \,\,& \exp\left(-\sum_{k=N}^{\infty}e^{-C k^{\b - \a}}\right),\text{ since }n \geq N_1\\
\geq \,\,& \frac{1-2\e}{1-\e},\text{ since }n\geq N_{2}.
\eae
Thus,
\bae
 & \mathbb{P}\left(D_{N}\geq N^{\beta},\,D_{N+1}=D_{N}+1,\,D_{N+2}=D_{N+1}+1,\,D_{N+3}=D_{N+2}+1,\,\cdots\right)\\
= \,\,& \mathbb{P}\left(D_{N}\geq N^{\beta}\right)\mathbb{P}\left(D_{N+1}=D_{N}+1,\,D_{N+2}=D_{N+1}+1,\,\cdots|\,D_{N}\geq N^{\beta}\right)\\
\geq \,\,& \left(1-\e\right)\frac{1-2\e}{1-\e} \,\, \geq \,\, 1-2\e,
\eae
which completes the proof of \eqref{eq:infgrowth}.

From \eqref{eq:infgrowth}, we have that for any $\e>0$, there exists $N_\e \in\mb N$ such that $$\mathbb{P}\left(\forall n\geq N_\e ,\,D_{n+1}=D_{n}+1\right)\geq1-2\e.$$
As $\left\{\forall n\geq N_\e,\,D_{n+1}=D_{n}+1 \right\} \subset \left\{ \exists\,N \in\mb N \textrm{ such that } \forall\,n\geq N,\,D_{n+1}=D_{n}+1  \right\}$,
$$\mb P\left(  \exists\,N \in\mb N \textrm{ such that } \forall\,n\geq N,\,D_{n+1}=D_{n}+1 \right) \geq 1-2\e .$$
Because we can take $\e>0$ arbitrary,
$$\mb P\left(  \exists\,N \in\mb N \textrm{ such that } \forall\,n\geq N,\,D_{n+1}=D_{n}+1 \right) = 1,$$
which completes the proof of Theorem~\ref{thm:MainThm}.

\bibliographystyle{plain}
\addcontentsline{toc}{section}{\refname}\bibliography{MCfromDLAbib}

\end{document}